\documentclass[reqno]{amsart}

\usepackage{tabu}
\usepackage{amssymb}
\usepackage{mathtools}
\usepackage{a4wide,amsmath}
\usepackage{mathrsfs}
\usepackage{amsthm}
\numberwithin{equation}{section}
\numberwithin{figure}{section}
\numberwithin{table}{section}
\usepackage{bbm}
\usepackage{subfig}
\usepackage{enumerate}
\usepackage[section]{placeins}
\usepackage{graphicx}		  
\usepackage{ifpdf}
\ifpdf
\DeclareGraphicsExtensions{.pdf,.eps,.jpg,.png}	
\usepackage[suffix=]{epstopdf}
\fi
\usepackage{xcolor}
\usepackage[utf8]{inputenc}
\usepackage{hyperref}
\hypersetup{hidelinks}

\long\def\MSC#1\EndMSC{\def\arg{#1}\ifx\arg\empty\relax\else
	{\narrower\noindent%
		{2020 Mathematics Subject Classification}: #1\\} \fi}
\long\def\PACS#1\EndPACS{\def\arg{#1}\ifx\arg\empty\relax\else
	{\narrower\noindent%
		{PACS numbers}: #1}\fi}
\long\def\KEY#1\EndKEY{\def\arg{#1}\ifx\arg\empty\relax\else
	{\narrower\noindent%
		Keywords: #1\\}\fi}

\newcommand{\Xint}[1]{\mathchoice
	{\XXint\displaystyle\textstyle{#1}}%
	{\XXint\textstyle\scriptstyle{#1}}%
	{\XXint\scriptstyle\scriptscriptstyle{#1}}%
	{\XXint\scriptscriptstyle\scriptscriptstyle{#1}}%
	\!\int}
\newcommand{\XXint}[3]{{\setbox0=\hbox{$#1{#2#3}{\int}$}
		\vcenter{\hbox{$#2#3$}}\kern-.5\wd0}}
\newcommand{\dashint}{\Xint-}

\theoremstyle{plain}
\newtheorem{theorem}{Theorem}[section]

\newtheorem{proposition}[theorem]{Proposition}

\theoremstyle{definition}

\newtheorem{assumption}[theorem]{Assumption}
\theoremstyle{remark}
\newtheorem{remark}[theorem]{Remark}

\newcommand{\norm}[1]{\lVert#1\rVert}
\newcommand{\abs}[1]{\lvert#1\rvert} 
\newcommand{\inner}[1]{\langle#1\rangle}

\newcommand{\redel}{\mathop{\textup{Re}}}

\newcommand{\dist}{\mathop{\textup{dist}}}
\newcommand{\essinf}{\mathop{\textup{ess\,inf}}}
\newcommand{\esssup}{\mathop{\textup{ess\,sup}}}

\newcommand{\suppm}{\mathop{\textup{supp}}}

\newcommand{\di}{\mathrm{d}}   

\begin{document}
	\title[Reconstruction of singular and degenerate inclusions]{Reconstruction of singular and degenerate inclusions in Calder\'on's problem}
	
	\author[H.~Garde]{Henrik Garde}
	\address[H.~Garde]{Department of Mathematics, Aarhus University, Ny Munkegade 118, 8000 Aarhus C, Denmark.}
	\email{garde@math.au.dk}
	
	\author[N.~Hyv\"onen]{Nuutti Hyv\"onen}
	\address[N.~Hyv\"onen]{Department of Mathematics and Systems Analysis, Aalto University, P.O. Box~11100, 00076 Helsinki, Finland.}
	\email{nuutti.hyvonen@aalto.fi}
	
	\begin{abstract}
		We consider the reconstruction of the support of an unknown perturbation to a known conductivity coefficient in Calder\'on's problem. In a previous result by the authors on monotonicity-based reconstruction, the perturbed coefficient is allowed to simultaneously take the values $0$ and $\infty$ in some parts of the domain and values bounded away from $0$ and $\infty$ elsewhere. We generalise this result by allowing the unknown coefficient to be the restriction of an $A_2$-Muckenhoupt weight in parts of the domain, thereby including singular and degenerate behaviour in the governing equation. In particular, the coefficient may tend to $0$ and $\infty$ in a controlled manner, which goes beyond the standard setting of Calder\'on's problem. Our main result constructively characterises the outer shape of the support of such a general perturbation, based on a local Neumann-to-Dirichlet map defined on an open subset of the domain boundary. 
	\end{abstract}	
	\maketitle
	
	\KEY
	Calder\'on's problem,
	electrical impedance tomography, 
	monotonicity method,
	inclusion detection,
	degenerate elliptic problem,
	singular elliptic problem.
	\EndKEY
	
	\MSC
	35R30, 35R05, 47H05.
	\EndMSC
	
	\section{Introduction} \label{sec:intro}
	
	Let $\Omega$ be a bounded Lipschitz domain in $\mathbb{R}^d$, $d\geq 2$, with connected complement and let $\Gamma\subseteq\partial\Omega$ be relatively open. For an appropriate nonnegative scalar-valued function $\gamma$, we consider the conductivity equation
	\begin{equation} \label{eq:conductivityeq}
		-\nabla\cdot(\gamma\nabla u) = 0 \quad\text{ in } \Omega.
	\end{equation}
    The task of determining $\gamma$ from Cauchy data of the solutions to \eqref{eq:conductivityeq} is called Calder\'on's inverse conductivity problem, which corresponds to the real-world imaging modality {\em electrical impedance tomography} (EIT); see~\cite{Borcea2002a,Borcea2002,Cheney1999,Uhlmann2009}. Instead of the general Calder\'on problem, this work concentrates on the simpler task of inclusion detection: Our aim is to reconstruct the outer shape of $\suppm(\gamma-\gamma_0)$ based on the local {\em Neumann-to-Dirichlet} (ND) map $\Lambda(\gamma)$ on $\Gamma$ for \eqref{eq:conductivityeq}, with $\gamma_0$ being a known background conductivity coefficient that belongs to
	\begin{equation*}
		L^\infty_+(\Omega) = \big\{ \varsigma\in L^\infty(\Omega;\mathbb{R}) \mid \essinf \varsigma > 0 \big\}
	\end{equation*}
	and satisfies a weak \emph{unique continuation principle} (UCP) in connection to the conductivity equation; cf.,~e.g.,~\cite[Definition~3.3]{Garde2020c}. The outer shape $D$ of $\suppm(\gamma-\gamma_0)$ refers to the smallest closed set in $\Omega$ having connected complement and satisfying $\suppm(\gamma-\gamma_0)\subseteq D$.
	
	In our analysis, the coefficient $\gamma$ may formally take values $0$ (perfectly insulating) and $\infty$ (perfectly conducting) in certain parts of $\Omega$; we refer to such regions as \emph{extreme inclusions}. If $\gamma$ is characterised by such extreme inclusions as well as finite positive and negative perturbations to $\gamma_0$, the monotonicity method has been proven to produce an exact reconstruction of the outer shape of $\suppm(\gamma-\gamma_0)$ under only mild geometric assumptions~\cite[Theorem~3.7]{Garde2020c}. However, this previous result requires that the restriction of $\gamma$ to the complement of the extreme inclusions belongs to $L^\infty_+$, thus only allowing (infinite) jump discontinuities at the boundaries of the extreme inclusions and no singular or degenerate behaviour elsewhere in $\Omega$.
	
	The purpose of this paper is to provide a simple extension to the proof of \cite[Theorem~3.7]{Garde2020c} so that it applies to a much larger class of unknown conductivity coefficients $\gamma$. To be more precise, we permit singular and degenerate behaviour by letting $\gamma$ be the restriction of an $A_2$-Muckenhoupt weight in parts of $\Omega$, which corresponds to a standard way of relaxing the assumptions on a coefficient in an elliptic partial differential equation \cite{Fabes1982,NLP}. Unlike in \cite[Theorem~3.7]{Garde2020c}, we also allow $\gamma$ to tend to $\gamma_0$ in a controlled manner rather than insisting on jump discontinuities at the boundary of $\suppm(\gamma-\gamma_0)$; see~\cite[Theorem~4.7]{Harrach13} for allowing such behaviour without extreme inclusions. Although there exist a few recent results on the unique solvability of the general Calder\'on problem with certain classes of singular and degenerate coefficients in two dimensions~\cite{Astala2016,Carstea2016,Nachman2020}, we are not aware of any previous extensions of inclusion detection methods to such frameworks. Moreover, we pose no restrictions on the spatial dimension $d \geq 2$ and only require local measurements.
	
	As a motivation for the use of $A_2$-coefficients, let us mention a couple of examples that are outside the standard setting for Calder\'on's problem but can still be tackled within our setting. An $A_2$-weight is allowed to locally behave as $\dist(\,\cdot\,,\Sigma)^s$ for any $s\in (-1,1)$ and Lipschitz hypersurface~$\Sigma$; see,~e.g.,~\cite[Lemma~3.3]{Duran2010}. In particular, such a hypersurface can be a subset of the boundary of an extreme inclusion, allowing a continuous decay to $0$ or a continuous growth to $\infty$ when approaching the inclusion boundary. Since even stronger singular or degenerate behaviour is possible for $A_2$-weights near a point, the coefficient $\gamma$ can also behave as $\dist(\,\cdot\,,x_0)^s$ close to an arbitrary $x_0 \in \Omega$ for any $s\in(-d,d)$.
	
	The main theoretical developments and related properties of monotonicity-based reconstruction in connection to Calder\'on's problem and EIT can be found in \cite{Garde2020c,Esposito2021,Garde_2019b,Harrach_2019,Harrach10,Harrach13,Ikehata1998a,Kang1997b,Tamburrino2002} for the continuum model and in \cite{GardeStaboulis_2016,Garde_2019,Harrach_2019,Harrach15} for  related practical electrode models. For an extensive list of references to papers employing similar monotonicity-based arguments in other inverse coefficient problems, we refer to \cite{Harrach_2019}. More general information on the theoretical aspects of Calder\'on's problem is available in the review article~\cite{Uhlmann2009} and the references therein.
	
	The rest of this article is organised as follows. Section~\ref{sec:main} introduces the main result on reconstruction of inclusions as Theorem~\ref{thm:general}; the required assumptions are summarised in Assumption~\ref{assump:D}. Section~\ref{sec:remarks} further elaborates on Assumption~\ref{assump:D}. Section~\ref{sec:forward} rigorously defines the Neumann problem for \eqref{eq:conductivityeq} and the associated local ND map in the context of a Muckenhoupt coefficient and extreme inclusions. Finally, Section~\ref{sec:proof} concludes the paper by proving Theorem~\ref{thm:general}.
	
	\section{Main result} \label{sec:main}
	
	The family of admissible inclusions is defined as 
	\begin{align*}
		\mathcal{A} &= \{ C \subset \Omega \mid C \text{ is the closure of an open set,}  \\
		&\hphantom{{}= \{C \subset \Omega \mid{}}\text{has connected complement,} \\
		&\hphantom{{}= \{C \subset \Omega \mid{}}\text{and has Lipschitz boundary } \partial C \}.
	\end{align*}
	Let $D\in\mathcal{A}$ be the set representing the inclusions in $\Omega$, that is, $D$ is (the outer shape of) the closure of the set where the investigated conductivity $\gamma$ differs from the known background conductivity $\gamma_0 \in L^\infty_+(\Omega)$. More precisely, $D$ is assumed to be composed of `negative' and `positive' parts as $D = D^- \cup D^+$,
	\begin{equation*}
		D^- = D_\textup{deg}\cup D_0 \cup D_\textup{F}^- \qquad \text{and} \qquad D^+ = D_\textup{sing} \cup D_\infty \cup D_\textup{F}^+,
	\end{equation*}
	where $D_\textup{deg}, D_\textup{sing}, D_0, D_\infty, D_\textup{F}^-, D_\textup{F}^+$ are mutually disjoint measurable sets, each of which may be empty or have multiple components with the exact conditions on their geometry given in Assumption~\ref{assump:D} below. 
	
	We define $\gamma$ as 
	\begin{equation*}
		\gamma = \begin{cases}
			0 & \text{in } D_0, \\
			\infty & \text{in } D_\infty, \\
			\gamma_\textup{deg} & \text{in } D_\textup{deg}, \\
			\gamma_\textup{sing} & \text{in } D_\textup{sing}, \\
			\gamma_\textup{F}^- & \text{in } D_\textup{F}^-, \\
			\gamma_\textup{F}^+ & \text{in } D_\textup{F}^+, \\
			\gamma_0 & \text{in } \Omega\setminus D.
		\end{cases}
	\end{equation*}
	Before introducing the exact assumptions on the different types of inclusions and the associated conductivities, we summarise that $D_{0}$ and $D_\infty$ correspond to the subsets of $\Omega$ where $\gamma$ is characterised by extreme inclusions and $D_{\textup{deg}}$ and $D_{\textup{sing}}$ to the subsets where $\gamma$ is allowed to be an $A_2$-Muckenhoupt weight. In the subsets $D_\textup{F}^\pm$ the coefficient $\gamma$ is bounded away from $0$ and $\infty$. However, it should be noted that in the nonextreme parts, $\gamma$ is only assumed to deviate from $\gamma_0$ near $\partial D$.

	Some implications and interpretations of the following assumptions are presented in Section~\ref{sec:remarks}.
	\begin{assumption} \label{assump:D}
		We assume the following about $D$ and $\gamma$.
		\begin{enumerate}[(i)]
			\item $\gamma \leq \gamma_0$ in $D^-$ and $\gamma \geq \gamma_0$ in $D^+$.
			\item $\gamma_\textup{F}^- \in L^\infty_+(D_\textup{F}^-)$ and $\gamma_\textup{F}^+ \in L^\infty_+(D_\textup{F}^+)$.
			\item\label{item:muckenhaupt} The sets $D_\textup{deg}$ and $D_\textup{sing}$ are compactly contained in the interior of $D$, and $\gamma_\textup{deg}$ and $\gamma_\textup{sing}$ are restrictions of an $A_2$-Muckenhoupt weight.
			\item \label{item:components} The sets $D_0, D_\infty, D_0\cup D_\textup{deg}\cup D_\textup{sing}$, and $D_\infty\cup D_\textup{deg}\cup D_\textup{sing}$ are closures of open sets with finitely many components and Lipschitz boundaries. Moreover, $D_0$ and $D_0\cup D_\textup{deg}\cup D_\textup{sing}$ have connected complements.
			\item\label{item:final} For every open neighbourhood $W$ of $x\in\partial D$, there exists a relatively open set $V\subset D$ that intersects $\partial D$, and $V\subset \widetilde{D} \cap W$ for one set $\widetilde{D}\in \{D_0, D_\infty, D_\textup{F}^-, D_\textup{F}^+\}$.
			\begin{enumerate}[(a)]
				\item If $\widetilde{D} = D_\textup{F}^-$, there exists an open ball $B\subset V$ such that $\esssup_{B}(\gamma_\textup{F}^- - \gamma_0) < 0$.
				\item If $\widetilde{D} = D_\textup{F}^+$, there exists an open ball $B\subset V$ such that $\essinf_{B}(\gamma_\textup{F}^+ - \gamma_0) > 0$.
			\end{enumerate}
		\end{enumerate}
	\end{assumption}
	
	Let $\Lambda(\gamma)$ be the local ND map on $\Gamma$ corresponding to the coefficient $\gamma$; see Section~\ref{sec:forward} for its precise definition, including an explanation on how extreme, singular, and degenerate inclusions are included in the associated elliptic Neumann boundary value problem. For $C\in\mathcal{A}$, we define $\Lambda_0(C)$ as the ND map with coefficient $0$ in $C$ and $\gamma_0$ outside $C$. In the same manner, we define $\Lambda_\infty(C)$ to be the ND map with coefficient $\infty$ in $C$ and $\gamma_0$ outside $C$.
	
	This prelude finally leads to the following general result on reconstructing inclusions from a local ND map.
	
	\begin{theorem} \label{thm:general}
		Let $\gamma$ and $D$ satisfy Assumption~\ref{assump:D} and $\gamma_0$ satisfy the UCP. For $C\in \mathcal{A}$, it holds
		\begin{equation*}
			D\subseteq C \qquad \text{if and only if} \qquad \Lambda_0(C)\geq \Lambda(\gamma) \geq \Lambda_\infty(C).
		\end{equation*}
		In particular, $D = \cap\{C\in\mathcal{A} \mid \Lambda_0(C)\geq \Lambda(\gamma)\geq \Lambda_\infty(C)\}$.
	\end{theorem}
	It follows immediately from the proof of Theorem~\ref{thm:general} that if $D^+$ is empty, then only the inequality $\Lambda_0(C)\geq \Lambda(\gamma)$ needs to be considered, and if $D^-$ is empty, then only  $\Lambda(\gamma)\geq \Lambda_\infty(C)$ needs to be considered.
	
	\section{Some remarks on Assumption~\ref{assump:D}} \label{sec:remarks}
	
	
	Let us present a few remarks on Assumption~\ref{assump:D}:
	
	\begin{itemize}
        \item The condition \eqref{item:muckenhaupt} ensures that at $\partial D$ the coefficient $\gamma$ has a jump to $0$ or $\infty$, or a well-behaving finite transition characterised by $\gamma_\textup{F}^-$ or $\gamma_\textup{F}^+$. In particular, singular and degenerate behaviour of an $A_2$-weight is not permitted exactly at the boundary of $D$ but only in its interior.  
		\item However, \eqref{item:muckenhaupt} still allows $\gamma$ to exhibit singular and degenerate behaviour in the interior of $D$,~e.g.,~by approaching $0$ or $\infty$ with a limited rate near the extreme inclusions, certain hypersurfaces, curves, or points.
        \item The condition \eqref{item:components} is related to technical assumptions for a convergence result in~\cite{Garde2020c} for potentials and ND maps in connection with extreme inclusions.
	    \item The condition \eqref{item:final} excludes certain pathological cases such as $\gamma-\gamma_0$ changing its sign arbitrarily often everywhere near an open part of $\partial D$. This type of oscillating behaviour is, however, allowed near points on $\partial D$ if these points are separated by at least a fixed positive distance. Indeed, the condition \eqref{item:final} ensures that one can access $D$ everywhere along $\partial D$ through an open set that only intersects with one of the inclusion types.
        \item However, \eqref{item:final} still allows $\gamma$ to approach $\gamma_0$ in a controlled manner near some parts of the inclusion boundary, while also allowing finite jump discontinuities as has previously been considered in connection to extreme inclusions in \cite{Garde2020c}. The former could,~e.g.,~correspond to $\gamma$ equalling $\gamma_0$ on an open subset $\Sigma\subset \partial D$ but exhibiting a local strict increase or decrease inside $D$ with respect to the distance from $\Sigma$.  
	\end{itemize}
	
	\section{Forward problem with a Muckenhoupt coefficient} \label{sec:forward}
	
	In this section we consider the Neumann problem for~\eqref{eq:conductivityeq} when the coefficient is a restriction of an $A_2$-Muckenhoupt weight that is well-behaved near $\partial \Omega$ and allows perfectly insulating and perfectly conducting parts in $\Omega$. We denote the coefficient in the Neumann problem by $\sigma$ to distinguish it from the fixed coefficient $\gamma$ in \eqref{eq:conductivityeq}. This provides the means to introduce the corresponding local ND map $\Lambda(\sigma)$, which in turn defines the forward map $\sigma \mapsto \Lambda(\sigma)$ associated to the considered inverse problem. 
	
	We refer to \cite{Fabes1982,NLP} for an introduction to Muckenhoupt weights and weighted Sobolev spaces, although for weighted Poincar\'e inequalities we refer to results in \cite{Drelichman2008} since those are shown for more general domains. A nonnegative function $w$ on $\mathbb{R}^d$ is called an $A_2$-Muckenhoupt weight provided that $w$ and $1/w$ are locally integrable and satisfy
	\begin{equation*}
		\exists C>0, \, \forall B \text{ open ball in $\mathbb{R}^d$}: \biggl( \dashint_B w\,\di x \biggr)\biggl( \dashint_B \frac{1}{w}\,\di x \biggr) \leq C.
	\end{equation*}
	A common equivalent definition integrates over cubes rather than balls. 
	
	Let $C_0\Subset \Omega$ be such that $\widetilde{\Omega} = \Omega\setminus C_0$ is a Lipschitz domain; $C_0$ will play the role of perfectly insulating inclusions in the following. For an $A_2$-weight $w$, define the norms
	\begin{align}
		\norm{v}_{L^2(\widetilde{\Omega},w)}^2 &= \int_{\widetilde{\Omega}} w\abs{v}^2 \,\di x, \label{eq:L2weight}\\[1mm]
		\norm{v}_{H^1(\widetilde{\Omega},w)}^2 &= \norm{v}_{L^2(\widetilde{\Omega},w)}^2 + \norm{\nabla v}_{L^2(\widetilde{\Omega},w)}^2, \label{eq:H1weight}
	\end{align}
        where $\norm{\nabla v}_{L^2(\widetilde{\Omega},w)}$ refers to the $L^2(\widetilde{\Omega},w)$-norm of $\abs{\nabla v}$. The weighted spaces $L^2(\widetilde{\Omega},w)$ and $H^1(\widetilde{\Omega},w)$ are then defined as the completions of the spaces of $C^\infty(\widetilde{\Omega})$-functions with finite norms with respect to \eqref{eq:L2weight} and \eqref{eq:H1weight}. In particular, both $L^2(\widetilde{\Omega},w)$ and $H^1(\widetilde{\Omega},w)$ are Hilbert spaces. 
	
	By density and using \cite[Theorem~3.3]{Drelichman2008} with $\alpha = 1$, $p=q=2$, and the weight function $\sqrt{w}$, we arrive at the weighted Poincar\'e inequality
	\begin{equation} \label{eq:Poincare}
		\inf_{c\in\mathbb{C}}\norm{v-c}_{L^2(\widetilde{\Omega},w)} \leq C\norm{\nabla v}_{L^2(\widetilde{\Omega},w)}, \qquad v\in H^1(\widetilde{\Omega},w).
	\end{equation}
	In particular, the quotient space $H^1(\widetilde{\Omega},w)/\mathbb{C}$ can be equipped with the norm 
	\begin{equation*}
		\norm{v}_{H^1(\widetilde{\Omega},w)/\mathbb{C}} = \norm{\nabla v}_{L^2(\widetilde{\Omega},w)}, \qquad v\in H^1(\widetilde{\Omega},w)/\mathbb{C},
	\end{equation*}
    which is equivalent to the standard quotient norm of $H^1(\widetilde{\Omega},w)/\mathbb{C}$ due to \eqref{eq:Poincare}.

	In order to define a local ND map, we need to ensure that the elements of the considered space $H^1(\widetilde{\Omega},w)$ have well-defined Dirichlet traces. To this end, suppose there exist $K\Subset \Omega$ and $c\in(0,1)$ such that $\widetilde{\Omega}\setminus K$ is a Lipschitz domain and $c\leq w \leq c^{-1}$ almost everywhere in $\widetilde{\Omega}\setminus K$. This is sufficient for guaranteeing the existence of a bounded Dirichlet trace map from $H^1(\widetilde{\Omega},w)$ to $H^{1/2}(\partial\Omega)$, and thus also to $L^2(\Gamma)$: First, notice that if $v\in H^1(\widetilde{\Omega},w)$, then there is a sequence $(\phi_i)$ in $C^\infty(\widetilde{\Omega})$ such that $\phi_i\to v$ in $H^1(\widetilde{\Omega},w)$. Consequently, it also holds that $\phi_i|_{\widetilde{\Omega}\setminus K} \to v|_{\widetilde{\Omega}\setminus K}$ in $H^1(\widetilde{\Omega}\setminus K,w) = H^1(\widetilde{\Omega}\setminus K)$. Hence, we may apply the standard trace theorem for $H^1(\widetilde{\Omega}\setminus K)$ to deduce
	\begin{equation*}
		\norm{v|_{\Gamma}}_{L^2(\Gamma)}\leq \norm{v|_{\partial\Omega}}_{H^{1/2}(\partial\Omega)} \leq C\norm{v}_{H^1(\widetilde{\Omega}\setminus K)} \leq C\norm{v}_{H^1(\widetilde{\Omega},w)}.
	\end{equation*}
	Moreover,
	\begin{equation*}
		\norm{v|_{\Gamma}}_{L^2(\Gamma)/\mathbb{C}} \leq C\norm{v}_{H^1(\widetilde{\Omega},w)/\mathbb{C}}, \qquad v\in H^1(\widetilde{\Omega},w)/\mathbb{C},
	\end{equation*}
    due to an obvious generalisation.
	
	Finally, let $C_\infty \Subset \Omega\setminus C_0$ be the closure of an open set with Lipschitz boundary; in our analysis, $C_\infty$ corresponds to perfectly conducting inclusions. We define the considered conductivity coefficient as
	\begin{equation*}
		\sigma = \begin{cases}
			0 &\text{in } C_0,\\
			\infty &\text{in } C_\infty,\\
			w &\text{in } \Omega\setminus(C_0\cup C_\infty)
		\end{cases}
	\end{equation*}
	and the corresponding closed subspace of $H^1(\widetilde{\Omega},w)/\mathbb{C}$ via
	\begin{equation*}
		\mathcal{H}(\sigma) = \left\{ v\in H^1(\Omega\setminus C_0,w)/\mathbb{C} \mid \nabla v = 0 \text{ in } C_\infty^\circ \right\},
	\end{equation*}
	where $C_\infty^\circ$ denotes the interior of $C_\infty$. We abbreviate this Hilbert space as $\mathcal{H}$ if there is no room for misinterpretation. The induced norm on $\mathcal{H}$ is
	\begin{equation*}
		\norm{v}_{\mathcal{H}}^2 = \int_{\Omega\setminus(C_0\cup C_\infty)} \sigma\abs{\nabla v}^2\,\di x, \qquad v\in \mathcal{H}.
	\end{equation*}
	
	For a current density $f$ belonging to the $\Gamma$-mean free space
	\begin{equation*}
		L^2_\diamond(\Gamma) = \{ g\in L^2(\Gamma) \mid \inner{g,1}_{L^2(\Gamma)} = 0 \},
	\end{equation*}
	we define the electric potential $u$, corresponding to the conductivity coefficient $\sigma$, as the unique solution in $\mathcal{H}$ to the variational problem
	\begin{equation} \label{eq:variational}
		\int_{\Omega\setminus(C_0\cup C_\infty)} \sigma\nabla u\cdot\overline{\nabla v}\,\di x = \inner{f,v|_\Gamma}_{L^2(\Gamma)}, \qquad \forall v\in \mathcal{H},
	\end{equation}
	where there is no ambiguity in the right hand-side because of the mean free condition for $f$. The unique solvability of \eqref{eq:variational} is a straightforward consequence of the Lax--Milgram lemma. We will occasionally write $u = u_f^\sigma$ if the connection of $u$ to the specific Neumann boundary value $f$ and conductivity coefficient $\sigma$ needs to be emphasised. 
	
	Since the left hand-side of \eqref{eq:variational} defines a symmetric sesquilinear form on $\mathcal{H}$, $u$ is also the unique minimiser of the following functional (cf.,~e.g.,~\cite[Remark~12.23]{Grubb} and~\cite[Theorem~1.1.2]{Ciarlet1978}):
	\begin{equation} \label{eq:energy}
		J_\sigma(v) = \int_{\Omega\setminus(C_0\cup C_\infty)} \sigma\abs{\nabla v}^2\,\di x - 2\redel\inner{f,v|_\Gamma}_{L^2(\Gamma)}, \qquad v\in \mathcal{H}.
	\end{equation}
	The dependence on $\sigma$ in the notation $J_\sigma$ also encodes the extreme inclusions $C_0$ and $C_\infty$, as well as the domain $\mathcal{H}$ for the functional.
	
	The corresponding local ND map $\Lambda(\sigma)\in\mathscr{L}(L^2_\diamond(\Gamma))$, in the space of bounded linear operators on $L_\diamond^2(\Gamma)$, is defined as $\Lambda(\sigma)f = u_{f}^\sigma|_\Gamma$. Here the notation is slightly abused by denoting with $u_{f}^\sigma$ the unique $\Gamma$-mean free element in the corresponding equivalence class of the quotient space $\mathcal{H}$, which makes the definition of $\Lambda(\sigma)$ concordant with the material in \cite{Garde2020c}. In particular, $\Lambda(\sigma)$ is compact, self-adjoint, and satisfies
	\begin{equation} \label{eq:ND}
		\inner{\Lambda(\sigma)f,f}_{L^2(\Gamma)} = \int_{\Omega\setminus(C_0\cup C_\infty)} \sigma\abs{\nabla u_f^\sigma}^2\,\di x, \qquad f\in L^2_\diamond(\Gamma),
	\end{equation}
    which are also standard properties of ND maps for conductivity coefficients in $L^\infty_+(\Omega)$.

    \begin{remark}
	    We have decided to use \eqref{eq:variational} as the definition of the Neumann problem for the conductivity equation with $A_2$-Muckenhoupt weights and extreme inclusions. For the connection of \eqref{eq:variational} to the underlying partial differential equation and associated Neumann and Dirichlet boundary conditions on $\partial C_0$ and $\partial C_\infty$, respectively, we refer to \cite{Garde2020c}, where the situation is analysed for standard $L^\infty_+$-coefficients. In particular, we consider the direct employment of the variational problem \eqref{eq:variational} as the natural generalisation to the case of $A_2$-weights.
    \end{remark}

	\section{Proof of Theorem~\ref{thm:general}} \label{sec:proof}
	
	As we will see, the proof can essentially be reduced to showing that $\Lambda(\gamma_\textup{L})\geq \Lambda(\gamma) \geq \Lambda(\gamma_\textup{U})$ for certain modified variants of $\gamma$ satisfying $\gamma_\textup{L} \leq \gamma \leq \gamma_\textup{U}$. Specifically, for $D_{\textup{A}_2} = D_\textup{deg}\cup D_\textup{sing}$ we define
	\begin{equation*}
		\gamma_\textup{L} = \begin{cases}
			0 & \text{in } D_{\textup{A}_2}, \\
			\gamma & \text{in } \Omega\setminus D_{\textup{A}_2},
		\end{cases} 
		\qquad \gamma_\textup{U} = \begin{cases}
			\infty & \text{in } D_{\textup{A}_2}, \\
			\gamma & \text{in } \Omega\setminus D_{\textup{A}_2}.
		\end{cases} 
	\end{equation*}
	Notice that $\gamma_\textup{L}$ and $\gamma_\textup{U}$ satisfy the conditions required by \cite[Theorem~3.7]{Garde2020c} if one modifies its statement and proof to also allow $\gamma$ to approach $\gamma_0$ in a controlled manner (cf.~Assumption~\ref{assump:D}\eqref{item:final}). This slight generalisation of \cite[Theorem~3.7]{Garde2020c} is given below.
	
	\begin{proposition} \label{prop1}
	  Let $\gamma$ and $D$ satisfy Assumption~\ref{assump:D}, but with $D_\textup{deg}=D_\textup{sing}=\emptyset$, and let $\gamma_0$ satisfy the UCP. For $C\in \mathcal{A}$, it holds
		\begin{equation*}
			D\subseteq C \qquad \text{if and only if} \qquad \Lambda_0(C)\geq \Lambda(\gamma) \geq \Lambda_\infty(C).
		\end{equation*}
	\end{proposition}
	
	\begin{proof}
		The first part of the proof for \cite[Theorem~3.7]{Garde2020c} remains identical in such a setting, whereas its second part must be slightly modified: when choosing the ball $B$ where potentials are localised in parts (a) and (b) of the proof, Assumption~\ref{assump:D}\eqref{item:final} guarantees that $B$ can be chosen such that $\gamma$ is bounded away from $\gamma_0$ in $B$, without requiring a jump discontinuity at the inclusion boundary as in the original assumptions for \cite[Theorem~3.7]{Garde2020c}.
	\end{proof}

    In what follows we denote $D_\textup{ext} = D_0 \cup D_\infty$. As the coefficients $\gamma_\textup{L}$ and $\gamma_\textup{U}$ purposefully do not have singular or degenerate parts, the corresponding potentials $u_f^{\gamma_\textup{L}}$ and $u_f^{\gamma_\textup{U}}$ are minimisers of the appropriate energy functionals defined by \eqref{eq:energy}, over the unweighted spaces $\mathcal{H}(\gamma_\textup{L})$ and $\mathcal{H}(\gamma_\textup{U})$, respectively. Moreover, $u_f^\gamma \in \mathcal{H}(\gamma)$ obviously satisfies $u_f^\gamma|_{\Omega\setminus (D_0\cup D_{\textup{A}_2})} \in \mathcal{H}(\gamma_\textup{L})$ and, on the other hand, $u_f^{\gamma_\textup{U}}  \in \mathcal{H}(\gamma)$ since it has vanishing gradient in the region where $\gamma$ may exhibit $A_2$ behaviour. These observations are employed below when manipulating the energy functionals associated to different conductivity coefficients.
        
	We start by proving $\Lambda(\gamma_\textup{L})\geq \Lambda(\gamma)$. Using \eqref{eq:variational} and \eqref{eq:ND}, we get
	\begin{align}
		-\inner{\Lambda(\gamma)f,f}_{L^2(\Gamma)} &= \int_{\Omega\setminus D_\textup{ext}} \gamma\abs{\nabla u_f^\gamma}^2\,\di x - 2\inner{f,u_f^\gamma|_\Gamma}_{L^2(\Gamma)} \notag\\
		&\geq \int_{\Omega\setminus (D_\textup{ext}\cup D_{\textup{A}_2})} \gamma_\textup{L}\abs{\nabla u_f^\gamma}^2\,\di x - 2\inner{f,u_f^\gamma|_\Gamma}_{L^2(\Gamma)} \notag\\
		&\geq \int_{\Omega\setminus (D_\textup{ext}\cup D_{\textup{A}_2})} \gamma_\textup{L}\abs{\nabla u_f^{\gamma_\textup{L}}}^2\,\di x - 2\inner{f,u_f^{\gamma_\textup{L}}|_\Gamma}_{L^2(\Gamma)} \notag\\[1mm]
		&= -\inner{\Lambda(\gamma_\textup{L})f,f}_{L^2(\Gamma)}, \label{eq:bndL}
	\end{align}
	where the $L^2(\Gamma)$ inner products have real values by \eqref{eq:variational} and \eqref{eq:ND}. The first inequality follows from the nonnegativity of the integrand and the fact that $\gamma$ equals $\gamma_\textup{L}$ outside $D_{\textup{A}_2}$, whereas the second one is a consequence of $u_f^{\gamma_\textup{L}}$ being the minimiser of $J_{\gamma_\textup{L}}$ in $\mathcal{H}(\gamma_\textup{L})$.
	
	Next we prove $\Lambda(\gamma)\geq \Lambda(\gamma_\textup{U})$:
	\begin{align}
		-\inner{\Lambda(\gamma_\textup{U})f,f}_{L^2(\Gamma)} &= \int_{\Omega\setminus (D_\textup{ext}\cup D_{\textup{A}_2})} \gamma_\textup{U}\abs{\nabla u_f^{\gamma_\textup{U}}}^2\,\di x - 2\inner{f,u_f^{\gamma_\textup{U}}|_\Gamma}_{L^2(\Gamma)} \notag\\
		&= \int_{\Omega\setminus D_\textup{ext}} \gamma\abs{\nabla u_f^{\gamma_\textup{U}}}^2\,\di x - 2\inner{f,u_f^{\gamma_\textup{U}}|_\Gamma}_{L^2(\Gamma)} \notag\\
		&\geq \int_{\Omega\setminus D_\textup{ext}} \gamma\abs{\nabla u_f^{\gamma}}^2\,\di x - 2\inner{f,u_f^{\gamma}|_\Gamma}_{L^2(\Gamma)} \notag\\[1mm]
		&= -\inner{\Lambda(\gamma)f,f}_{L^2(\Gamma)}, \label{eq:bndU}
	\end{align}
	where the second equality follows from $\abs{\nabla u_f^{\gamma_\textup{U}}}$ vanishing in $D_{\textup{A}_2}$, while the inequality is a consequence of $u_f^{\gamma}$ being the minimiser of $J_\gamma$ in $\mathcal{H}(\gamma)$.
	
	To prove the actual theorem, assume first $D\subseteq C$. Due to \eqref{eq:bndL}, \eqref{eq:bndU}, and Proposition~\ref{prop1} applied to  $\gamma_\textup{L}$ and $\gamma_\textup{U}$,
	\begin{equation}
          \label{eq:chainOpIneq}
		\Lambda_0(C) \geq \Lambda(\gamma_\textup{L}) \geq \Lambda(\gamma) \geq \Lambda(\gamma_\textup{U}) \geq \Lambda_\infty(C),
	\end{equation}
    which proves the direction ``$\Rightarrow$'' of the assertion.
	
    To prove the opposite direction ``$\Leftarrow$'', assume that $D\not \subseteq C$, i.e.,\ there is a part of $D\setminus C$ that can be connected to $\Gamma$ via a relatively open connected set $U\subset \Omega\setminus C$. By virtue of Assumption~\ref{assump:D}\eqref{item:final}, we may assume that $U$ only intersects one of the inclusion types $D_0, D_\infty, D_\textup{F}^-$, or $D_\textup{F}^+$. Proposition~\ref{prop1} does not directly reveal which one of its two operator inequalities that fails when applied to $\gamma_\textup{L}$ or $\gamma_\textup{U}$ in case $D\not \subseteq C$ (cf.~\eqref{eq:chainOpIneq}). However, according to \cite[Proof of Theorem~3.7]{Garde2020c}, this depends on the inclusion type that $U$ intersects, and we choose our approach to dealing with $\gamma$ accordingly: If (Case~A) the intersected inclusion type belongs to $D^-$, we will investigate $\Lambda_0(C) - \Lambda(\gamma)$, and, on the other hand, if (Case~B) the inclusion type belongs to $D^+$, we will investigate $\Lambda(\gamma) - \Lambda_\infty(C)$.
	
	We start by considering Case~A, meaning that the only part of $D$ that $U$ is chosen to intersect is either $D_0$ or $D_\textup{F}^-$. According to \eqref{eq:bndU},
	\begin{equation*}
		\Lambda_0(C) - \Lambda(\gamma_\textup{U}) \geq \Lambda_0(C) - \Lambda(\gamma).
	\end{equation*}
	However, by repeating either part (b) (if $U$ intersects $D_\textup{F}^-$) or part (d) (if $U$ intersects $D_0$) in \cite[Proof of Theorem~3.7]{Garde2020c}, we deduce $\Lambda_0(C)\not\geq \Lambda(\gamma_\textup{U})$, and therefore also $\Lambda_0(C)\not\geq \Lambda(\gamma)$.
	
	In Case~B, the only part of $D$ that $U$ is chosen to intersect is either $D_\infty$ or $D_\textup{F}^+$. According to~\eqref{eq:bndL},
	\begin{equation*}
		\Lambda(\gamma_\textup{L}) - \Lambda_\infty(C) \geq \Lambda(\gamma) - \Lambda_\infty(C).
	\end{equation*}
	However, by repeating either part (a) (if $U$ intersects $D_\textup{F}^+$) or part (c) (if $U$ intersects $D_\infty$) in \cite[Proof of Theorem~3.7]{Garde2020c}, we deduce $\Lambda(\gamma_\textup{L}) \not\geq \Lambda_\infty(C)$, and therefore also $\Lambda(\gamma) \not\geq \Lambda_\infty(C)$. \hfill\qed
    
    \subsection*{Acknowledgments}
    
    This work is supported by the Academy of Finland (decision 336789) and the Aalto Science Institute (AScI). In addition, HG is supported by The Research Foundation of DPhil Ragna Rask-Nielsen and is associated with the Aarhus University DIGIT Centre, and NH is supported by Jane and Aatos Erkko Foundation via the project Electrical impedance tomography --- a novel method for improved diagnostics of stroke.
    
	\bibliographystyle{plain}

\end{document}